\documentclass[preprint,12pt]{elsarticle}

\usepackage{amssymb}
\usepackage{amsthm}
\usepackage{dsfont}
\usepackage{amsmath}
\usepackage{amsmath,tikz,wasysym,multirow} 
\usepackage{enumerate}

\newtheorem{proposition}{Proposition}
\newtheorem{example}[proposition]{Example}
\newtheorem{lemma}[proposition]{Lemma}

\newtheorem{theorem}[proposition]{Theorem}

\newcommand{\npmatrix}[1]{\left( \begin{matrix} #1 \end{matrix} \right)} 
\newcommand{\R}{\mathbb{R}}

\newcommand{\mc}{\mathcal}

\journal{Linear Algebra Appl.}

\begin{document}

\begin{frontmatter}



\title{Completely positive factorizations associated with Euclidean distance matrices corresponding to an arithmetic progression}


\author[lab1,lab2]{Damjana Kokol Bukov\v{s}ek\fnref{footnote}}
\fntext[footnoste]{Damjana Kokol Bukov\v{s}ek acknowledges financial support from the Slovenian Research Agency (research core funding No. P1-0222). The work on the paper was done during her three-month visit to University College Dublin.}
\author[lab3]{Thomas Laffey}
\author[lab3]{Helena \v{S}migoc}
\address[lab1]{School of Economics and Business, University of Ljubljana, Slovenia}
\address[lab2]{Institute of Mathematics Physics and Mechanics, Ljubljana, Slovenia}
\address[lab3]{School of Mathematics and Statistics, University College Dublin, Ireland}

\begin{abstract}
Euclidean distance matrices corresponding to an arithmetic progression have rich spectral and structural properties. We exploit those properties to develop completely positive factorizations of translations of those matrices. We show that the minimal translation that makes such a matrix positive semidefinite results in a completely positive matrix. We also discuss completely positive factorizations of such matrices over the integers. Methods developed in the paper can be used to find completely positive factorizations of other matrices with similar properties. 
\end{abstract}




\end{frontmatter}

\section{ Introduction }

An $n\times n$ real symmetric matrix $A$ is \emph{completely positive}, if it can be written as $A=BB^T$ for some $n \times k $ entry-wise non-negative matrix $B$. The minimum number of columns $k$ in such a non-negative factor $B$ is called the \emph{cp-rank} of $A$. The set of completely positive matrices forms a closed, convex cone, and better understanding of this cone is not only an interesting question, but it also plays an important role in the field of \emph{copositive optimization}, \cite{MR1810401, 10.1007/978-3-642-12598-0_1, MR2892529}. 

It is immediate from the definition, that any completely positive matrix is non-negative and positive semidefinite. A matrix that is both non-negative and positive semidefinite is said to be \emph{doubly non-negative}. Not every $n \times n$ doubly non-negative matrix of order $n \geq 5$ is completely positive, and deciding if a given matrix is completely positive is a hard problem that has been a topic of intense research over the years. For a comprehensive survey on completely matrices we refer to \cite{MR1986666}. 

Research has shown that the question can be solved for different classes of matrices. Here we mention two representative results that are different in flavour. In \cite{MR910984} it is shown that every diagonally dominant non-negative matrix is completely positive. Kogan and Berman \cite{MR1217760} introduced a graph theoretic view on the problem. For a given $n\times n$ symmetric matrix $A$, let $G(A)$ be a graph on $n$ vertices with an edge $(i,j)$ if and only if $a_{ij} \neq 0$.  A graph $G$ is said to be \emph{completely positive} if every doubly non-negative matrix with $G(A)=G$ is completely positive. In \cite{MR1217760} it is shown that a graph $G$ is completely positive if and only if it does not contain an odd cycle of length greater than $4$.

Finding a factorization of a given completely positive matrix is a natural question in this context. A factorization algorithm that can factorize any matrix in the interior of the cone of completely positive matrices is offered in \cite{MR3152068}, explicit factorization of diagonally dominant matrices is presented in \cite{MR910984}, and generalized to matrices whose comparison matrix is positive definite in \cite{MR1310974}. Factorization of matrices with conditions on their graph is studied for example in \cite{MR923678,MR3023447}. Those factorizations are typically not optimal, the number of columns in the factorization matrix is larger than the minimal possible. 
Finding and optimizing the cp-rank adds additional complexity to an already hard problem. 

Once we know that a factorization of a given matrix $A$ exists, we may want to impose further properties on the factors. A matrix $A$ has a rational cp-factorization, if it can be decomposed as $BB^T$ where $B$ the entries of $B$ are non-negative and rational. Integer cp-factorization is defined in a similar way.  Every rational matrix which lies in the interior of the cone of completely positive matrices has a rational cp-factorization \cite{MR3624664}. Integer question seems to be harder to study, and has only recently been answered for $2 \times 2$ matrices \cite{MR3904097}. For a detailed state-of-the art account on those two questions we refer the reader to \cite{MR3859579}.

In this paper we study completely positive factorizations of translations of Euclidean distance matrices corresponding to an arithmetic progression. Distance matrices have rich structural and spectral properties, and have been extensively studied in the literature \cite{MR3887551,MR787367, MR3152147}. Our work was particularly motivated by results on factorizations of distance matrices \cite{MR2964717, MR2653832, MR3918551}. In this paper we develop further factorizations of translations of distance matrices  corresponding to an arithmetic progression, that depend on specific properties of this class. 
 Methods, that we introduce, can be adapted to other matrices with similar properties. While Euclidean distance matrices are well behaved in many ways, they are not positive definite. The first factorization that  we develop has a factor with one negative element, and subsequent factorizations are developed for translations of  Euclidean distance matrices. In Section \ref{sec:optimal} we find the best possible result for this question. Finally, in Section \ref{sec:integer} we present some factorizations, where the factors have integer entries. 

We will use standard notation in the paper. The $n \times n$ identity matrix will be denoted by $I_n$, and the $n \times n$ nilpotent Jordan block by $J_n$. $B(i)$ will denote the matrix obtained from $B$ by deleting the $i$-th row and column. More generally, for $\mathcal I \subseteq \{1,2,\ldots,n\}$, we will denote by $B(\mathcal I)$ the matrix obtained from $B$ by deleting the row and columns indexed by $\mathcal I$. 

\section{Euclidean Distance Matrices corresponding an arithmetic progression}

For a set of distinct real numbers $\mathcal S = \{a_1, a_2, ..., a_n\}$, we  define {\it the Euclidean distance matrix} to be the matrix $\mathrm{EDM}(\mathcal S)=(a_{ij})_{i,j=1}^n$ with $a_{ij} = (a_j - a_i)^2$. Such a matrix $\mathrm{EDM}(\mathcal S)$ is clearly entrywise non-negative, and it has rank $3$, see \cite{MR2563025}.

In this work we are interested in a family of Euclidean distance matrices that correspond to an arithmetic progression, i.e. $\mathcal S = \{a, a+d, ..., a+(n-1) d\}$ for some positive real numbers $a$ and $d$, and in particular for  $\mathcal S = \{1, 2, ..., n\}$. We will denote $$A_n := \mathrm{EDM}(\{1, 2, ..., n\}),$$ hence $(A_n)_{ij} = (j-i)^2.$
We can limit our discussion to this special case, as for $\mathcal S = \{a, a+d, ..., a+(n-1) d\}$ we have $EDM(\mc S) = d^2 A_n $, and all the properties of $A_n$ are easily adapted to this more general case. In our first result we explicitly compute the eigenvalues and the eigenvectors of $A_n$. 

\begin{lemma}\label{lem:eigenvalues}
The nonzero eigenvalues of $A_n$ are equal to:
$$ \lambda_{1,2} = \textstyle\frac{1}{12}n(n^2-1) \pm \sqrt{\textstyle\frac{1}{240}n^2(n^2-1)(3n^2-7)}, 
\lambda_3 = -\textstyle\frac{1}{6}n(n^2-1),$$
$\lambda_1 > 0 > \lambda_2 > \lambda_3$. The eigenvector corresponding to the eigenvalue $\lambda_3$ is equal to 
$ w = (w_i)_{i=1}^n$ with $w_i:=n+1-2i$ for $i=1,\ldots, n$. For $j=1,\ldots, n-3$ let $v(j)$ denote the vector with only four nonzero elements, defined as follows: $v(j)_j=1$, $v(j)_{j+1}=-3$, $v(j)_{j+2}=3$, $v(j)_{j+3}=-1$. Then the vectors $v(j)$, $j=1,\ldots,n-3$, form a basis of the null space of $A_n$.  
\end{lemma}

\begin{proof}
We can prove that $w$ is an eigenvector for $A_n$ corresponding to $\lambda_3$ by direct computation:
$$ (A_nw)_i = \sum_{j=1}^n (j-i)^2(n+1-2j) = -\textstyle\frac{1}{6}n(n^2-1)(n+1-2i).$$
It is also straightforward to check that $A_nv(j)=0$, for $j=1,\ldots,n-3$. 
Since rank of $A_n$ is $3$, we still need to compute two more nonzero eigenvalues. 


To compute the remaining two eigenvalues, we note that the trace of $A_n$ is equal to zero, i.e. $\lambda_1+\lambda_2+\lambda_3=0$, 
and we compute the trace of $A_n^2$: 
$$ \lambda_1^2 + \lambda_2^2 + \lambda_3^2 = \sum_{i=1}^n \sum_{j=1}^n (j-i)^4 = \textstyle\frac{1}{30}n^2(n^2-1)(2n^2-3).$$
Taking into account that $\lambda_3$ is known, we obtain the following quadratic equation for $\lambda_1$ and $\lambda_2$:
$$ 180 \lambda^2 - 30 n(n^2-1) \lambda - n^2(n^2-1)(n^2-4) = 0 .$$
Solving this equation gives us: $$ \lambda_1 = \textstyle\frac{1}{12}n(n^2-1) + \sqrt{\textstyle\frac{1}{240}n^2(n^2-1)(3n^2-7)}$$
and $$ \lambda_2 = \textstyle\frac{1}{12}n(n^2-1) - \sqrt{\textstyle\frac{1}{240}n^2(n^2-1)(3n^2-7)}.$$
Finally, the inequality $\lambda_1>\lambda_2 > \lambda_3$ is straightforward to check.
\end{proof}

Since $A_n$ has two negative eigenvalues it is clearly not completely positive. Still, the aim of this work is to develop factorizations of matrices $A_n$ and $A_n+\alpha I_n$ that take into account not only the nonnegativity of $A_n$, but also as many of the following properties of $A_n$ as possible: symmetry, integer entries, pattern, and low rank of $A_n$. We won't be able to attend to all these properties with a single factorization, and we will allow translations of $A_n$ to make it positive semidefinite.  

\section{Straightforward factorizations related to $A_n$}

First we offer a factorization of $A_n$, that is reminiscent of a completely positive factorization, but has an additional factor with a negative element. 

\begin{theorem}\label{thm:LRL}
We can write $A_n=L_nRL_n^T$, where $L_n$ is a non-negative $n \times 3$ matrix and
$$R = \npmatrix{ 0 & 1 & 1 \\ 1 & -6 & 1 \\ 1 & 1 & 0 }.$$
\end{theorem}

\begin{proof}
Notice that the vectors $v(j)$, $j=1,\ldots,n-3$, defined in Lemma \ref{lem:eigenvalues} are the first $n-3$ columns of the matrix $(I_n-J_n^T)^{3}$.
 It follows that first $n-3$ columns of the matrix 
$A_n(I_n-J_n^T)^{3}$ are equal to zero, and, by symmetry, the first $n-3$ rows of $\hat A_n:=(I_n-J_n)^{3}A_n(I_n-J_n^T)^{3}$ are also zero. 
Hence,  only the lower-right $3\times 3$ corner of $\hat A_n$ is non-zero, and it is straightforward to check that this corner is equal to 
$$R = \npmatrix{ 0 & 1 & 1 \\ 1 & -6 & 1 \\ 1 & 1 & 0 }.$$

It follows that $A_n=(I_n-J_n)^{-3}(0_{n-3}\oplus R)(I_n-J_n^T)^{-3}$, and $L_n$ is equal to the last three columns of $(I_n-J_n)^{-3}$. The claim is proved by noting that $(I_n-J_n)^{-1} = I_n + J_n + J_n^2 + ... + J_n^{n-1}$ is a non-negative matrix. 
\end{proof}

An interesting feature of the factorization in Theorem \ref{thm:LRL} is that the matrix $R$ is independent of $n$. Since we cannot hope to find a completely positive factorization of $A_n$, we focus on the factorizations of matrices of the form $$A_n+g(n) I_n.$$ Clearly, $A_n+g(n) I_n$ will become diagonally dominant, and hence completely positive, for all large enough $g(n)$. The minimal $g(n)$ for which $A_n$ is diagonally dominant is easily computed, and is equal to
$g_D(n):=\sum_{j=1}^{n-1} j^2=\frac{1}{6}n(n-1)(2n-1).$ On the other hand, the minimal $g(n)$ for which $A_n+g(n) I_n$ will become positive semidefinite can be deduced from Lemma \ref{lem:eigenvalues}:  $f(n)= \frac{1}{6}n(n^2-1)$. The minimal $g(n)$ for which $A_n+g(n)I_n$ is completely positive therefore satisfies the following inequality: 
\begin{equation}\label{eq:bound1}
\frac{1}{6}n(n-1)(n+1) \le g(n) \le \frac{1}{6}n(n-1)(2n-1).
\end{equation}
The main result of this paper will show that the lower bound is achieved, but first we will improve the upper bound in \eqref{eq:bound1} using an inductive approach. 
Informed by \eqref{eq:bound1} we denote $f(n)=\frac{1}{6}n(n-1)(n+1)$, and look for functions $g(n)$ of the form $g(n)=q f(n)$, for which we can prove that $A_n+g(n)I_n$ is completely positive using a straightforward inductive proof. 

Note that $A_{n-1}=A_n(1)=A_n(n)$, and the most natural inductive approach to the problem is to write   
$$A_n + g(n)I_n=(A_{n-1}+g(n-1)I_{n-1}) \oplus 0_1 +R_n,$$
where $v_{n-1}=\npmatrix{(n-1)^2 & (n-2)^2 & \ldots & 1}^T$ and $$R_n=\npmatrix{(g(n)-g(n-1)) I_{n-1} & v_{n-1} \\ v_{n-1}^T & g(n)}.$$
If we already know that $A_{n-1}+g(n-1)I_{n-1}$ is completely positive by inductive hypothesis, then it is sufficient to prove that $R_n$ is completely positive to determine the complete positivity of $A_n$. Since the graph of $R_n$ contains no odd cycles of order greater than $3$, the matrix $R_n$ is completely positive as soon as it is doubly non-negative. Note that to determine that $R_n$ is positive semidefinite we only need to show that $\det R_n >0$ from consecutive minors result, see for example \cite[Theorem 7.2.5]{MR2978290}. Inserting $g(n)=q f(n)$ into 
 $$\det R_n=\left(g(n)-g(n-1)\right)^{n-2}\left(g(n)(g(n)-g(n-1))-v_{n-1}^Tv_{n-1}\right),$$
 and collecting $(g(n)(g(n)-g(n-1))-v_{n-1}^Tv_{n-1})$ in terms of $n$, gives us $q= 2
   \sqrt{\frac{3}{5}}$ as the optimal choice for $q$ for this approach. 
  
Next we consider an induction step that takes us from $n-2$ to $n$. Note that $A_{n-2}=A_n(\{1,2\})=A_n(\{n-1,n\})=A_n(\{1,n\})$, and it turns out that developing our induction process from $A_{n-2}=A_n(\{1,n\})$ gives the best bound. 
Hence, we write 
 $$A_n + g(n)I=0_1 \oplus \left( A_{n-2}+g(n-2)I_{n-2}\right) \oplus 0_1 + Q_n,$$
where 
$$Q_n = \npmatrix{g(n) & u_{n-2}^T & (n-1)^2 \\ u_{n-2} & (g(n)-g(n-2))I_{n-2} & v_{n-2} \\ (n-1)^2 & v_{n-2}^T & g(n) },$$
$u_{n-2} = \npmatrix{ 1 & 2^2 & ... & (n-2)^2}^T,$ and  $v_{n-2} =\npmatrix{ (n-2)^2 & ... & 2^2 & 1}^T$.
As above, it is sufficient to consider the complete positivity of $Q_n$. The graph of $Q_n$ contains no odd cycles of order greater than $3$, so $Q_n$ will be completely positive as soon as it will be 
positive semidefinite. To study the positive definiteness of $Q_n$, we consider a matrix that is permutationally similar to $Q_n$:
 $$\hat Q_n=\npmatrix{ (g(n)-g(n-2))I_{n-2} & u_{n-2} & v_{n-2} \\ 
                                       u_{n-2}^T & g(n) & (n-1)^2 \\
                                       v_{n-2}^T & (n-1)^2 & g(n)}.$$
Assuming $g(n) >g(n-2)$, $\hat Q_n$ will be positive semidefinite if and only if $\det \hat Q_n(n) >0$ and $\det \hat Q_n\geq 0$. Using Schur complement to compute those determinants, and considering the leading coefficient in $n$, we determine that $\det \hat Q_n(n) >0$ for all $q > \sqrt{\frac{6}{5}}$, and that $\det \hat Q_n>0$ if and only if 
 $$\det\npmatrix{\delta_n g(n)-u_{n-2}^Tu_{n-2} & \delta_n(n-1)^2-v_{n-2}^Tu_{n-2} \\ \delta_n(n-1)^2-v_{n-2}^Tu_{n-2} & \delta_n g(n)-v_{n-2}^Tv_{n-2}}>0,$$
where $\delta_n=(g(n)-g(n-2))$. Taking into account:
\begin{align*}
u_{n-2}^Tu_{n-2} =v_{n-2}^Tv_{n-2}&=\frac{1}{30}  (n-1)(n-2) (2 n-3) \left(3 n^2-9 n+5\right),  \\
v_{n-2}^Tu_{n-2}&=\frac{1}{30} n(n-1)(n-2) \left(n^2-2 n+2 \right),
\end{align*}
and $g(n)=q f(n)$, we determine that $q= \sqrt{\frac{7}{5}}$ is optimal for this method. This gives us the following proposition that strengthens the bound given in  \eqref{eq:bound1}. 

\begin{proposition}\label{prop:bound2}
The matrix $A_n+\sqrt{\frac{7}{5}} f(n)I_n$ is completely positive. 
\end{proposition}

\section{ Optimal $g(n)$ }\label{sec:optimal}

At this point we would like to strengthen Proposition \ref{prop:bound2} to $q=1$. To this end we denote $B_n=A_n+f(n)I_n$. Note that $\det B_n=0$, hence we are operating at the boundary of positive semidefinite matrices. The following simple and well known result points out an additional property that the factors in the cp factorization of a semidefinite matrix must satisfy. 

\begin{proposition}\label{prop:nullspace}
Let $A=BB^T$ be a completely positive matrix with a nontrivial kernel $\mathcal N$. Then $B^Tv=0$ for every $v \in \mathcal N$.
\end{proposition} 

The proof of the main result in this section will be rather technical, as we will explicitly develop a completely positive factorization of $B_n$. We will aim to write $B_n$ as the sum of rank one completely positive matrices, and a matrix with particular structure that is made explicit in the following lemma. 

\begin{lemma}\label{special}
Let $D_1 \in M_m(\R)$ and $D_2 \in M_k(\R)$ be diagonal matrices, and let $C$ be an $m \times k$ non-negative matrix. Then any matrix of the form
$$ A =  \npmatrix{D_1 & C \\ C^T & D_2} \in M_{m+k}(\R),$$
with an eigenvector $w=\npmatrix{w_1^T & -w_2^T}^T$ corresponding to a non-negative eigenvalue, where $w_1 \in M_{m}(\R)$ and $w_2 \in M_{k}(\R)$ are strictly positive, is completely positive. 
\end {lemma}

\begin{proof}
First we observe that the matrix $B=(I_m \oplus -I_k)A(I_m \oplus -I_k)$ has nonpositive off-diagonal elements, and a positive eigenvector $\hat w=\npmatrix{w_1^T & w_2^T}^T$ corresponding to a non-negative eigenvalue. We deduce that $B\hat w\geq 0$, and hence $B$ is an $M$-matrix, see for example \cite{MR1298430}. Since $B$ is also symmetric, it is positive semidefinite. This, in particular, proves that $D_1$ and $D_2$ are non-negative. Since $A$ is similar to $B$, we now know that $A$ is doubly non-negative. Moreover, the graph of $A$ is bipartite, hence $A$ is completely positive, as it is proved in \cite{MR923678}. 
\end{proof}

Proposition \ref{prop:nullspace} and Lemma \ref{special} bring to light, why it is convenient to know the singular vector of $B_n$ explicitly when we are developing the completely positive factorization of $B_n$. 
In what follows let us denote the eigenvector corresponding to the eigenvalue $\lambda_3$ given in Lemma \ref{lem:eigenvalues} as $w(n)$, or just $w$, when $n$ is clear from the context. (Clearly, $w(n)$ is a singular vector for $B_n$.) Part of the proof will be split into an even and an odd case. One of the reasons for this is the slight difference in the pattern of $w(n)$ in those cases as $w(2m)=\npmatrix{v_{2m},-v_{2m}'}^T$ and $w(2m+1)=\npmatrix{v_{2m+1},0,v_{2m+1}'}^T$, where $v_n=\npmatrix{n-1 & n-3 & \ldots & n-(2 \lfloor \frac{n}{2}\rfloor-1) }^T$ and $v_n'=\npmatrix{ n-(2 \lfloor \frac{n}{2}\rfloor-1)& \ldots & n-3 & n-1 }^T$.

\begin{lemma}\label{lem:antidiag}
Let $K_n \in M_n(\R)$ be the matrix with ones on the anti-diagonal and zeros elsewhere. Then $K_n^2=I_n$, $K_nA_nK_n=A_n$ and $K_nv_n=v_n'$. 
\end{lemma}

\begin{lemma}\label{even}
Let $n=2m$. There exists a non-negative matrix $U$ satisfying $U^Tw(n)=0$, diagonal matrices $D, D' \in M_m(\R)$ and a non-negative matrix $C \in M_m(\R)$, such that
$$ B_n = UU^T + \npmatrix{D & C \\ C^T & D'}.$$
\end {lemma}

\begin{proof} 
As suggested by the result we will view $B_n$ as a $2 \times 2$ block matrix: 
$$ B_n = \npmatrix{ A_m + f(n)I_m & H_m \\ H_m^T & A_m + f(n)I_m},$$
where $H_m=(h_{ij}) \in M_m(\R)$ is defined by
$h_{ij}=(m+j-i)^2$. We will construct $U$ by finding two matrices $V$ and $V'$ satisfying 
 $UU^T=VV^T+V'V'^T$,
 where each column of $V$ corresponds to an off-diagonal pair of entries in the upper-left block of $B_n$ and each column of $V'$ corresponds to an off-diagonal pair of entries in the lower-right block of $B_n$. 

Explicitly, $V:=\sum_{i=1}^{m-1}\sum_{j=i+1}^{m} u_{ij} u_{ij}^T,$
where 
\begin{align*}
u_{ij} &:= (j-i) (e_i + e_j + \alpha_{ij} e_{m+k}), \,
 k:= \textstyle \lfloor\frac{m+j}{2}\rfloor + 1 - i, \text{ and }\\ 
\alpha_{ij} &:= \textstyle \frac{2}{2k-1}(2m-i-j+1).
\end{align*}
Note that the matrix $u_{ij} u_{ij}^T$ has exactly nine nonzero entries, three on the main diagonal, one at the position $(i,j)$, two in the upper-right block and symmetrically three more below the main diagonal. The entry in $(i,j)$ position is equal to $(i-j)^2$, and hence equal to the $(i,j)$-th entry of $B_n$. Furthermore, once $k$ is chosen, the value of $\alpha_{ij}$ is determined from the condition $u_{ij}^Tw(n)=0$. (While there is some choice in what $k$ can be, we need to make sure that we choose it in such a way that the off-diagonal blocks remain non-negative after subtraction from $B_n$.) Thus the matrix $VV^T$ has the form
$$ VV^T = \npmatrix{A_m + E_1 & S \\ S^T & E_2 },$$
where $E_1$ and $E_2$ are diagonal matrices, and satisfies the condition $V^Tw(n)=0$. 

We define $V':=K_nV$. Using block partition $$K_n=\npmatrix{0_m & K_m \\ K_m & 0_m}$$ and Lemma \ref{lem:antidiag} it is easy to determine that
$$V'V'^T= \npmatrix{E_1' & S' \\ S'^T & A_m+E_2' } \text{ and } V'^Tw(n)=0,$$
where $E_1'=K_mE_2K_m$, $E_2'=K_mE_1K_m$ and $S'=K_mS^TK_m$. In particular, 
$S'_{xy}=S_{m+1-y,m+1-x}.$


To complete the proof, we need to show that the matrix $C=H_m-S-S'$ is non-negative. 
 Let us first consider the $(x, y)$-entry in $S$: $s_{xy}=\sum_{1 \leq i < j \leq m}(u_{ij}u_{ij}^T)_{x,m+y}$.
 The matrix $u_{ij}u_{ij}^T$ has a nonzero entry at $(x, m+y)$ if and only if $(x, y) = (i, k)$ or $(x, y) = (j, k)$.
In the first case we have $i=x$ and $y = \lfloor\frac{m+j}{2}\rfloor + 1 - i$, thus $j= 2(x+y-1)-m$ or $j= 2(x+y-1)+1-m$.
In the second case we have $j=x$ and $i = \lfloor\frac{m+x}{2}\rfloor + 1 - y$. So, at most three 
columns of matrix $VV^T$ contribute to the entry $s_{xy}$. We will refer to this contributions as matrices $S_I$, $S_{II}$ and $S_{III}$, whose nonzero entries are defined as follows: 
\begin{align*}
(S_I)_{xy}&=(u_{x,2(x+y-1)-m}u_{x,2(x+y-1)-m}^T)_{x,y+m} \text{ for }1 \le x < 2(x+y-1)-m \le m \\
(S_{II})_{xy}&=(u_{x,2(x+y-1)+1-m}u_{x,2(x+y-1)+1-m}^T)_{x,y+m} \text{ for }1 \le x < 2(x+y-1)+1-m \le m \\
(S_{III})_{xy}&=(u_{\lfloor\frac{m+x}{2}\rfloor+1-y,x}u_{\lfloor\frac{m+x}{2}\rfloor+1-y,x}^T)_{x,y+m} \text{ for }1 \le \lfloor\frac{m+x}{2}\rfloor+1-y < x \le m.
\end{align*}
%
The computation of relevant entries of $u_{ij}u_{ij}^T$ and reordering of the bounds, gives us the following formulas for the nonzero entries of $S_{I}, S_{II}$ and $S_{III}$:
\begin{align*}
(S_I)_{xy} &= \textstyle
\frac{2}{2y-1}(x+2y-m-2)^2(3m-3x-2y+3) \text{ for }  m+3-2y \le x \le m+1-y, \\
(S_{II})_{xy} &= \textstyle\frac{2}{2y-1}(x+2y-m-2)^2(3m-3x-2y+2) \text{ for }  m+2-2y \le x \le m-y, \\ 
(S_{III})_{xy} &=\textstyle\frac{2}{2y-1}(x-\lfloor\frac{m+x}{2}\rfloor+y-1)^2(2m-\lfloor\frac{m+x}{2}\rfloor+y-x) \\
&\text{ for } \max\{m+3-2y,2y-m-1\} \le x \le m.
\end{align*}
Note that outside the indicated regions the entries of those matrices are equal to zero.
To simplify the calculations we estimate the matrix $S_{III}$ by a matrix $\hat S_{III}$ with nonzero entries in positions $(x,y)$ satisfying $m+3-2y \le x \le m$ and defined by: 
$$(S_{III})_{xy} \le (\hat S_{III})_{xy} = \textstyle
\frac{2}{2y-1}(x-\frac{m+x-1}{2}+y-1)^2(2m-\frac{m+x-1}{2}+y-x).$$

Since the matrices $H_m$ and $S + S'$ are symmetric with respect to the counter-diagonal it is enough to prove the nonnegativity of elements $(H_m-S-S')_{xy}$ that satisfy $x+y \le m+1$. Inside this region we will consider four subregions, that are determined by the location of nonzero entries of matrices $S_I$, $S_{II}$, $S_{III}$, $S'_I$, $S'_{II}$ and $S'_{III}$. Before we define the regions, we observe that $(S_{II})_{xy}=0$ for $x+y\geq m+1$, hence $(S'_{II})_{xy}=0$ for $x+y\leq m+1$ and $S'_{II}$ doesn't need to be considered. Similarly, $(S_I)_{xy}=0$ for $x+y \ge m+2$, so $(S'_{I})_{xy}$ only contribution is when $x+y=m+1$. At this point we also note that $(S)_{m1}=(S')_{m1}=0$. Now let us define and study each of the four regions:
\begin{enumerate}[{Region} (1):] 

\item $\{(x,y);\, x \le m+1-2y\}$, only $(\hat S'_{III})_{xy} \neq 0$, hence $C_{xy}=(H_m-S'_{III})_{xy} \ge (H_m-\hat S'_{III})_{xy}$.  Introducing new variables $u=y-1\ge 0$ and $v=m+1-x-2y\ge0$ gives us:
$$ C_{xy} \ge \frac{3 + 55 u + 129 u^2 + 81 u^3 + 2 v + 52 u v + 66 u^2 v + 12 u v^2}{4 (3 + 4 u + 2 v)} > 0.$$

\item $\{(x,y);\, x = m+2-2y, y \ge 2)\}$, $(S_{II})_{xy} \neq 0$ and $(\hat S'_{III})_{xy} \neq 0$, hence $C_{xy} \ge (H_m-S_{II}-\hat S'_{III})_{xy}$, which we express in terms of $u=y-2\ge 0$. We get
$$ C_{xy} \ge \frac{320 + 1184 u + 1558 u^2 + 855 u^3 + 162 u^4}{4 (3 + 2 u) (5 + 4 u)} > 0.$$

\item $\{(x,y);\, m+3-2y \le x \le m-y \}$, $(S_{I})_{xy} \neq 0$, $(S_{II})_{xy} \neq 0$, $(\hat S_{III})_{xy} \neq 0$, $(\hat S'_{III})_{xy} \neq 0$, and $C_{xy} \ge (H_m-S_I-S_{II}-\hat S_{III}-\hat S'_{III})_{xy}$.
 Change of variables $u=x+2y-m-3\ge 0$ and $v=m-x-y\ge0$ gives us
$$ C_{xy} \ge \frac{p(u,v)}{4 (5 + 2 u + 2 v) (7 + 2 u + 4 v)} > 0,$$
where
\begin{align*}
p(u,v) &= 1261 + 1487 u + 601 u^2 + 122 u^3 + 12 u^4 + 3629 v + 3306 u v \,+ \\
       &+\, 849 u^2 v + 72 u^3 v + 3565 v^2 + 2351 u v^2 + 312 u^2 v^2 + 1371 v^3 \,+ \\
       &+\, 516 u v^3 + 162 v^4.
\end{align*}

\item $\{(x,y); \, x=m+1-y, y \geq 2 \}$, $(S_{I})_{xy}=(S'_{I})_{xy} \neq 0$, $(\hat S_{III})_{xy}=(\hat S'_{III})_{xy} \neq 0$, and $C_{xy}\ge (H_m-2S_I-2 \hat S_{III})_{xy}$, 
which we express in terms of $u=y-2\ge 0$ to get
$$ C_{xy} \ge \frac{6 + 16 u + 12 u^2 + 3 u^3}{2 (3 + 2 u)} > 0.$$
\end{enumerate}

\end{proof}

\begin{lemma}\label{odd}
Let $n=2m+1$. There exists a non-negative matrix $U$ satisfying $U^Tw(n)=0$, $\alpha \ge 0$, diagonal matrices 
$D, D' \in M_m(\R)$ and a non-negative matrix $C \in M_m(\R)$, such that
$$ B_n = UU^T + \npmatrix{D & 0 & C \\ 0 & \alpha & 0 \\ C^T & 0 & D'}.$$
\end {lemma}

%
%
%

\begin{proof}
The proof will go along the same lines as the proof of Lemma \ref{even}. This time we will view $B_n$ as a $3 \times 3$ block matrix: 
$$ B_n = \npmatrix{ A_m + f(n)I_m & b_m & H_m \\ b_m^T & f(n) & (b'_m)^T \\  H_m^T & b'_m & A_m + f(n)I_m },$$
where $b_{m}=\npmatrix{m^2 & (m-1)^2 & \ldots & 1}^T$, $b'_m=K_mb_m$ and $H_m=(h_{ij}) \in M_m(\R)$ is defined by
$h_{ij}=(m+1+j-i)^2$. We will write $U$ as: 
 $$UU^T=VV^T+V'V'^T+tt^T+ZZ^T,$$
where (as in the proof of Lemma \ref{even}) each column of $V$ will correspond to an off-diagonal pair of entries in the upper-left block of $B_n$, and each column of $V'$ will correspond to an off-diagonal pair of entries in the lower-right block of $B_n$. In this case we will need to consider separately the position $(m-1, m)$ in the upper-left block, and  the position $(m+2, m+3)$ in the lower-right block. To deal with those positions we introduce a rank one matrix $tt^T$. Finally,  each column of $Z$ will correspond to an entry in the upper-middle block of $B_n$.

As in the even case, we define
$V:=\sum_{i=1}^{m-2}\sum_{j=i+1}^{m} u_{ij} u_{ij}^T,$
where 
\begin{align*}
u_{ij} &:= (j-i) (e_i + e_j + \alpha_{ij} e_{m+1+k}), \,
 k:= \textstyle \lfloor\frac{m+j+1}{2}\rfloor - i, \\ 
\alpha_{ij} &:= \textstyle \frac{1}{k}(2m-i-j+2),
\end{align*}
and $V':=K_nV$.
There are two positions in the diagonal blocks left to cover, and this is done by $tt^T$, where:
$t := e_{m-1}+e_m+e_{m+2}+e_{m+3}.$ It is easy to verify that $V^Tw(n)=0$, $V'^Tw(n)=0$, $t^Tw(n)=0$, and
$$VV^T+V'V'^T+tt^T= \npmatrix{A_m + \hat E_1 & 0 & \hat S \\ 0 & 0 & 0 \\ \hat S & 0 & A_m+\hat E_2 },$$
where $\hat E_1$ and $\hat E_2$ are diagonal matrices. Finally, we define 
$Z:=\sum_{i=1}^m z_i z_i^T,$ where
$$z_i := (m+1-i) (e_i + e_{m+1} + e_{2m+2-i}).$$
We have $z_i^Tw_n=0$, and 
$$ ZZ^T = \npmatrix{F & b_m & FK_m \\ b_m^T & \beta & (b'_m)^T  \\ K_mF & b'_m & K_mFK_m },$$
where $F$ is a diagonal matrix.

To complete the proof, we need to show that $\alpha:= f(n)- \beta > 0$, and that the matrix $C:=H_m-\hat S-FK_m$ is non-negative. 
We have 
$$\beta = \sum_{i=1}^m i^2 = \frac{1}{6}m(m+1)(2m+1) < f(n) = \frac{2}{3}m(m+1)(2m+1),$$
proving $\alpha >0$.

The matrix $tt^T$ influences only the lower-left $2 \times 2$ corner of $C$. For $n=5$ we have
$$C= \npmatrix{8 & 11 \\ 2 & 8},$$ and for $n \ge 7$ the lower-left $2 \times 2$ corner of $C$ is equal to $$\npmatrix{0 & 6 \\ 2 & 0}\ge 0.$$

The $(x,y)$ entry of the matrix $FK_m$ is nonzero only when $x+y=m+1$, and in this case $(FK_m)_{xy} = y^2$.

We will consider $VV^T$ and $V'V'^T$ separately, and to this end we denote by $S$ the top-right corner of $VV^T$ and by $S'=K_mS^TK_m$ the top-right corner of $V'V'^T$. Contributions of $S$ and $S'$ are dealt with as in Lemma \ref{even}, with only minor changes in indexing. 
Considering nonzero entries of the matrix $u_{ij}u_{ij}^T$ at position $(x, m+1+y)$, we get $(x, y) = (i, k)$ or $(x, y) = (j, k)$. This results in three possibilities, as follows.  
First we can have $i=x$ and $j= 2x+2y-1-m$, and we define the corresponding matrix $S_I$ by 
$$(S_I)_{xy}=(u_{x,2x+2y-1-m}u_{x,2x+2y-1-m}^T)_{x,m+1+y}$$ for $1 \le x < 2x+2y-1-m \le m, x\le m-2,$
and zero for $(x, y)$ outside this  range. Similarly, for  $i=x$ and  $j= 2x+2y-m$ we have $S_{II}$ with nonzero entries  
$$(S_{II})_{xy}=(u_{x,2x+2y-m}u_{x,2x+2y-m}^T)_{x,m+1+y}$$ for $1 \le x < 2x+2y-m \le m, x\le m-2$, and for  $j=x$ and $i = \lfloor\frac{m+x+1}{2}\rfloor - y$ we define $S_{III}$ with nonzero entries 
$$(S_{III})_{xy}=(u_{\lfloor\frac{m+x+1}{2}\rfloor-y,x}u_{\lfloor\frac{m+x+1}{2}\rfloor-y,x}^T)_{x,m+1+y}$$  for $ 1 \le \lfloor\frac{m+x+1}{2}\rfloor-y < x \le m, \lfloor\frac{m+x+1}{2}\rfloor-y \le m-2$.


Computation of relevant entries of $u_{ij}u_{ij}^T$ and reordering of the bounds, gives us the following formulas for the nonzero entries of $S_{I}, S_{II}$ and $S_{III}$:
\begin{align*}
(S_I)_{xy} &= \textstyle
\frac{1}{y}(x+2y-m-1)^2(3m-3x-2y+3) \text{ for }  m+2-2y \le x \le m-y, \\
(S_{II})_{xy} &= \textstyle\frac{1}{y}(x+2y-m)^2(3m-3x-2y+2) \text{ for }  m+1-2y \le x \le m-y, y \ge 2, \\ 
(S_{III})_{xy} &=\textstyle\frac{1}{y}(x-\lfloor\frac{m+x+1}{2}\rfloor+y)^2(2m-\lfloor\frac{m+x+1}{2}\rfloor+y-x+2) \\
&\text{ for } \max\{m+2-2y,2y-m\} \le x \le m, (x, y)\ne(m, 1).
\end{align*}
We estimate the matrix $S_{III}$ by a matrix $\hat S_{III}$, $\hat S_{III}\geq S_{III}$, with nonzero entries in positions $(x,y)$ satisfying $m+2-2y \le x \le m, (x, y)\ne(m, 1)$ and defined by: 
$$(S_{III})_{xy} \le (\hat S_{III})_{xy} = \textstyle
\frac{1}{y}(x-\frac{m+x}{2}+y)^2(2m-\frac{m+x}{2}+y-x+2).$$

The matrices $H_m$, $S + S'$ and $FK_m$ are symmetric with respect to the counter-diagonal, so it is enough to prove the nonnegativity of elements $(H_m-FK_m-S-S')_{xy}$ that satisfy $x+y \le m+1$. Inside this region we will consider four subregions, that are determined by the location of nonzero entries of matrices $S_I$, $S_{II}$, $S_{III}$, $S'_I$, $S'_{II}$, $S'_{III}$ and $FK_m$. Since $(S'_{I})_{xy}=(S'_{II})_{xy}=0$ for $x+y\leq m+1$, matrices $S'_{I}$ and $S'_{II}$ don't need to be considered. The four regions are defined as follows:
\begin{enumerate}[{Region} (1):] 
\item $\{(x,y);\, x \le m-2y\}$, only $(S'_{III})_{xy} \neq 0$, hence $C_{xy}=(H_m-S'_{III})_{xy} \ge (H_m-\hat S'_{III})_{xy}$.  Introducing new variables $u=y-1\ge 0$ and $v=m-x-2y\ge0$ gives us:
$$ C_{xy} \ge \frac{24 + 220 u + 258 u^2 + 81 u^3 + 8 v + 104 u v + 66 u^2 v + 12 u v^2}{8 (3 + 2 u + v)} > 0.$$
\item $\{(x,y);\, x = m+1-2y, y \ge 2\}$, $(S_{II})_{xy} \neq 0$ and $(S'_{III})_{xy} \neq 0$, hence $C_{xy} \ge (H_m-S_{II}-\hat S'_{III})_{xy}$, which we express in terms of $u=y-2\ge 0$. We get
$$ C_{xy} \ge \frac{305 + 691 u + 435 u^2 + 81 u^3}{16 (2 + u)} > 0.$$
\item $\{(x,y);\, m+2-2y \le x \le m-y\}$, $(S_{I})_{xy} \neq 0$, $(S_{II})_{xy} \neq 0$, $(S_{III})_{xy} \neq 0$, $(S'_{III})_{xy} \neq 0$, and $C_{xy} \ge (H_m-S_I-S_{II}-\hat S_{III}-\hat S'_{III})_{xy}$.
 Change of variables $u=x+2y-m-2\ge 0$ and $v=m-x-y\ge0$ gives us
$$ C_{xy} \ge \frac{p(u,v)}{8 (2 + u + v) (3 + u + 2 v)} ,$$
where
\begin{align*}
p(u,v) &= -122 - 71 u + 25 u^2 + 30 u^3 + 6 u^4 + 361 v + 568 u v + 241 u^2 v \,+ \\
       &+\, 36 u^3 v + 909 v^2 + 825 u v^2 + 156 u^2 v^2 + 531 v^3 + 258 u v^3 +  81 v^4.
\end{align*}
We get $C_{xy} > 0$ unless $(u, v) \in \{ (0, 0), (1, 0)\}$. For $(u, v) \in \{ (0, 0), (1, 0)\}$ we get $(x, y) \in \{(m-2,2), (m-3, 3)\}$. In those two cases we compute $C_{m-2,2} = 3 >0$ and $C_{m-3,3} = \frac{11}{4} >0$.
\item $\{(x,y); \, x=m+1-y, y \geq 3 \}$, $(FK_m)_{xy} \neq 0$, $(S_{III})_{xy}=(S'_{III})_{xy} \neq 0$, and $C_{xy}\ge (H_m-FK_m-2 \hat S_{III})_{xy}$, 
which we express in terms of $u=y-3\ge 0$ to get
$$ C_{xy} \ge \frac{68 + 116 u + 52 u^2 + 7 u^3}{4 (3 + u)} > 0.$$
\end{enumerate}

\end{proof}

\begin{theorem}\label{optimal}
The matrix $B_n = A_n + f(n)I$ is completely positive for every $n$.
\end {theorem}

\begin{proof}
Suppose first that $n=2m$ is even. By Lemma \ref{even} we write $B_n$ as 
$$B_n = UU^T + \npmatrix{D & C \\ C^T & D'},$$
and since the matrix $\npmatrix{D & C \\ C^T & D'}$ satisfies conditions of Lemma \ref{special}, the proof is complete in this case. 

Suppose now that $n=2m+1$. By Lemma \ref{odd} we can write $B_n$ as 
$$B_n = UU^T + \npmatrix{D & 0 & C \\ 0 & \alpha & 0 \\ C^T & 0 & D'}.$$
Now $\npmatrix{D & C \\ C^T & D'}$ is completely positive by Lemma \ref{special}, since its eigenvector corresponding to $0$ is just $w(2m+1)$ with the middle entry (which is equal to zero) omitted. In addition, we have proved that $\alpha > 0$, and this allows us to conclude that
$$\npmatrix{D & 0 & C \\ 0 & \alpha & 0 \\ C^T & 0 & D'}$$ is completely positive.
\end{proof}

\section{ Integer completely positive factorization }\label{sec:integer}

In this section we want to take into account that $A_n$ has integer entries, and we ask the question, for what values of $g(n)$ does the matrix $A_n + g(n)I$ have an integer completely positive factorization. It turns out, that for small $n$,  $g(n)=f(n)$ works, but for $n\geq 6$ we need to choose $g(n) >f(n)$.

\begin{proposition}\label{prop:smalln}
Let $B_n:=A_n+f(n)I_n$. Then:
\begin{enumerate}
\item $B_n$ has an integer completely positive factorization for $n \le 5$.
\item$B_6$ does not have an integer completely positive factorization.
\item $B_6+I_6$ has an integer completely positive factorization.
\end{enumerate}
\end {proposition}

\begin{proof}
To prove the first item and the third item, we find integer completely positive factorizations explicitly, as follows: 
\begin{align*}
 B_2 &= \npmatrix{ 1 & 1 \\ 1 & 1 } = \npmatrix{ 1 \\ 1 } \npmatrix{ 1 & 1 }, 
 \end{align*}
 \begin{align*}
B_3 &= \npmatrix{ 4 & 1 & 4 \\ 1 & 4 & 1 \\ 4 & 1 & 4 } = 
\npmatrix{ 1 \\ 1 \\ 1 } \npmatrix{ 1 & 1 & 1 } + 
3 \npmatrix{ 1 & 0 \\ 0 & 1 \\ 1 & 0 } \npmatrix{ 1 & 0 & 1 \\ 0 & 1 & 0} ,\\
\end{align*}
 \begin{align*}
B_4 &= \npmatrix{ 10 & 1 & 4 & 9 \\ 1 & 10 & 1 & 4 \\ 4 & 1 & 10 & 1 \\ 9 & 4 & 1 & 10 } = 
\npmatrix{ 1 & 1 & 0 \\ 1 & 0 & 3 \\ 1 & 3 & 0 \\ 1 & 0 & 1 } 
\npmatrix{ 1 & 1 & 1 & 1 \\ 1 & 0 & 3 & 0 \\ 0 & 3 & 0 & 1 } + 
8 \npmatrix{ 1 \\ 0 \\ 0 \\ 1 } \npmatrix{ 1 & 0 & 0 & 1 } , \\
\end{align*}
 \begin{align*}
B_5 &= \npmatrix{ 
20 & 1 & 4 & 9 & 16 \\ 1 & 20 & 1 & 4 & 9 \\ 4 & 1 & 20 & 1 & 4 \\ 9 & 4 & 1 & 20 & 1 \\ 16 & 9 & 4 & 1 & 20 } = 
\npmatrix{ 1 & 2 & 0 & 0 \\ 1 & 0 & 4 & 0 \\ 1 & 0 & 0 & 4 \\ 1 & 4 & 0 & 0 \\ 1 & 0 & 2 & 0 } 
\npmatrix{ 1 & 1 & 1 & 1 & 1 \\ 2 & 0 & 0 & 4 & 0 \\ 0 & 4 & 0 & 0 & 2 \\ 0 & 0 & 4 & 0 & 0 } + \\
&+ 3 \npmatrix{ 1 & 0 & 2 \\ 0 & 1 & 0 \\ 1 & 0 & 0 \\ 0 & 1 & 0 \\ 1 & 0 & 2 } 
\npmatrix{ 1 & 0 & 1 & 0 & 1 \\ 0 & 1 & 0 & 1 & 0 \\ 2 & 0 & 0 & 0 & 2 } .
\end{align*}
Furthermore: 
 \begin{align*}&B_6+I_6 =A_6+36 I_6=
 \npmatrix{ 
36 & 1 & 4 & 9 & 16 & 25 \\ 1 & 36 & 1 & 4 & 9 & 16 \\ 4 & 1 & 36 & 1 & 4 & 9 \\ 9 & 4 & 1 & 36 & 1 & 4 \\ 16 & 9 & 4 & 1 & 36 & 1 \\ 25 & 16 & 9 & 4 & 1 & 36  
} =\\
 &= \npmatrix{ 1 \\ 1 \\ 1 \\ 1 \\ 1 \\ 1 }\npmatrix{ 1 & 1 & 1 & 1 & 1 & 1 } + 
2 \npmatrix{ 1 & 0 & 0 \\ 0 & 0 & 2 \\ 0 & 4 & 0 \\ 4 & 0 & 0 \\ 0 & 0 & 2 \\ 0 & 1 & 0 }
\npmatrix{ 1 & 0 & 0 & 4 & 0 & 0 \\ 0 & 0 & 4 & 0 & 0 & 1 \\ 0 & 2 & 0 & 0 & 2 & 0 } + \\ 
 &+ 3 \npmatrix{ 1 & 0 \\ 0 & 1 \\ 1 & 0 \\ 0 & 1 \\ 1 & 0 \\ 0 & 1 } 
\npmatrix{ 1 & 0 & 1 & 0 & 1 & 0 \\ 0 & 1 & 0 & 1 & 0 & 1 } +
6 \npmatrix{ 1 & 0 & 2 \\ 0 & 2 & 0 \\ 0 & 0 & 0 \\ 0 & 0 & 0 \\ 2 & 0 & 0 \\ 0 & 1 & 2 }
\npmatrix{ 1 & 0 & 0 & 0 & 2 & 0 \\ 0 & 2 & 0 & 0 & 0 & 1 \\ 2 & 0 & 0 & 0 & 0 & 2 } .
\end{align*}

To prove the second item, we assume that  
$$B_6= \npmatrix{ 
35 & 1 & 4 & 9 & 16 & 25 \\ 1 & 35 & 1 & 4 & 9 & 16 \\ 4 & 1 & 35 & 1 & 4 & 9 \\ 9 & 4 & 1 & 35 & 1 & 4 \\ 16 & 9 & 4 & 1 & 35 & 1 \\ 25 & 16 & 9 & 4 & 1 & 35  
}$$
has an integer completely positive factorization: $B_6 = UU^T$. Let $u_i$, $i=1,\ldots,r$, denote the columns of $U$. First we note, that all columns of $U$ need to be orthogonal to
$w = \npmatrix{ 5 & 3 & 1 & -1 & -3 & -5 }^T$ by Proposition \ref{prop:nullspace}.  
Next we look at conditions that are coming from the fact that the superdiagonal of $B_6$ has all elements equal to $1$. In particular, this implies, that if two consecutive entries of any column of $U$ are nonzero, they both have to be equal to $1$.  
Moreover, since $(B_6)_{12} = 1$, one 
of the columns in $U$ has to have the first two entries both equal to one, and, without loss of generality,  we may assume that this holds for the first column of $U$: 
$u_1 = \npmatrix{ 1 & 1 & * & * & * & * }^T .$ Taking into account all the conditions on $u_1$ that we have listed so far, $u_1$ has to be equal to one of the following three vectors: 
\begin{align*}
v &= \npmatrix{ 1 & 1 & 1 & 1 & 1 & 1 }^T ,\\
v' &= \npmatrix{ 1 & 1 & 0 & 0 & 1 & 1 }^T , \\
v'' &= \npmatrix{ 1 & 1 & 0 & 3 & 0 & 1 }^T . 
\end{align*}
In the first case we have 
$$B_6 - u_1u_1^T = U_1U_1^T = \npmatrix{ 
34 & 0 & 3 & 8 & 15 & 24 \\ 0 & 34 & 0 & 3 & 8 & 15 \\ 3 & 0 & 34 & 0 & 3 & 8 \\ 8 & 3 & 0 & 34 & 0 & 3 \\ 15 & 8 & 3 & 0 & 34 & 0 \\
24 & 15 & 8 & 3 & 0 & 34
} .$$

If $u_1=v'$, then some column of $U$ (suppose $u_2$) has $1$ in the second and the third position. This forces the first entry of $u_2$ to be zero, and the condition $w^Tu_2=0$ leaves us with only one choice for $u_2$: $u_2 = \npmatrix{ 0 & 1 & 1 & 1 & 1 & 0 }^T ,$ thus 
$$B_6 - u_1u_1^T - u_2u_2^T = U_2U_2^T = \npmatrix{ 
34 & 0 & 3 & 8 & 15 & 24 \\ 0 & 33 & 0 & 2 & 7 & 15 \\ 3 & 0 & 33 & 0 & 2 & 8 \\ 8 & 2 & 0 & 33 & 0 & 3 \\ 15 & 7 & 2 & 0 & 33 & 0 \\
24 & 15 & 8 & 3 & 0 & 34
} .$$

If $u_1=v''$, we notice that some column of $U$ (suppose $u_2$) has $1$ in the last two positions, and since the first two entries of $u_2$ cannot both be equal to one, we deduce that
$u_2 = \npmatrix{ 1 & 0 & 3 & 0 & 1 & 1 }^T .$ Since $(B_6)_{23}=1$, we still need a column, say $u_3$, in $U$ that has $1$ in both the second and the third position. Again, we are left with only one option:
$u_3 = \npmatrix{ 0 & 1 & 1 & 1 & 1 & 0 }^T ,$ and
$$B_6 - u_1u_1^T - u_2u_2^T - u_3u_3^T = U_3U_3^T = \npmatrix{ 
33 & 0 & 1 & 6 & 15 & 23 \\ 0 & 33 & 0 & 0 & 8 & 15 \\ 1 & 0 & 25 & 0 & 0 & 6 \\ 6 & 0 & 0 & 25 & 0 & 1 \\ 15 & 8 & 0 & 0 & 33 & 0 \\
23 & 15 & 6 & 1 & 0 & 33
} .$$

In all three cases we are left with a matrix $U_i U_i^T$ that has all the entries on the superdiagonal equal to zero. In particular, 
the columns of $U_i$ have at least every other entry equal to zero. Moreover, one of the columns in $U_i$ has to have the first and the third
entry nonzero, so it is of the form:
$\hat v= \npmatrix{ a & 0 & b & 0 & c & d }^T ,$
where $cd=0$. 
Looking at the values of $U_iU_i^T$ in positions $(1,3)$, $(3,5)$ and $(2,6)$ we get $ab\leq 3$, $bc \leq 3$, and  $bd \leq 8$. This, together with the condition  $w^T\hat v=0$, gives us only one option for $\hat v$:
$$\hat v = \npmatrix{ 1 & 0 & 1 & 0 & 2 & 0 }^T .$$
In the third case we also need $bc=0$, so this choice immediately leads to contradiction. In the remaining two cases $U_iU_i^T - \hat v\hat v^T$ still has the $(1,3)$-entry positive, so we need to have another column in $U_i$, $i=1,2$,  of the same form. 
However, in both cases $U_iU_i^T-2 vv^T$ has a negative entry in $(3,5)$ position, leading to a contradiction. 
\end{proof}

The ad hoc approach, that we used to find completely positive factorizations in the proof above, is not well suited for generalisation to matrices of arbitrary size. Instead, we develop a more general technique, that can be applied to a class of Toeplitz matrices, but that does not produce the optimal $q$ for distance matrices. 

\begin{lemma}\label{lem:Ei}
Let $n \in \mathbb{N}$, $i \in \{1,\ldots,n-1\}$, and $$ E_i := I + J_n^i + J_n^{2i} + ... + J_n^{(n-1)i} + (J_n^T)^i + (J_n^T)^{2i} + ... + (J_n^T)^{(n-1)i} .$$
Any matrix of the form $\sum_{i=1}^{n-1} a_i E_i$, $a_i \in \mathbb{N}\cup \{0\}$, has an integer completely positive factorization. 
\end{lemma}

\begin{proof}
Clearly, it is sufficient to prove the statement for $E_i$, $i=1,\ldots,n-1$. So,  let us choose $i \in \{1,\ldots,n-1\}$, and write
$n = qi+r$ where $r \in \{0,\ldots,i-1\}$. We define: 
$$ U_i^T = \npmatrix{ I_i & I_i & ... & I_i & D_r },$$
where $D_r$ is an $i \times r$ (possibly null) matrix with $1$'s on the main diagonal and $0$'s elsewhere. Equality $E_i=U_iU_i^T$ can be checked by a straightforward calculation. 
\end {proof}

Before we see how this lemma can be applied to distance matrices, we illustrate its application by an example.

\begin{example}
The matrix of the form $$E_1 + E_2 + \ldots + E_{n-1} = (n-1)I + \sum_{i=1}^{n-1} \tau(i)(J_n^i + (J_n^T)^i) $$
has an integer completely positive factorization, where 
$\tau(i) := \sum_{d | i} d^0$, the number of all divisors of $i$.
\end{example}


The previous example shows a straightforward application of Lemma \ref{lem:Ei}, where $a_i$ are all chosen to be one. To determine the values for $a_i$, $i=1,\ldots,n-1$, so that the matrices $A_n$ and $\sum_{i=1}^{n-1} a_i E_i$ agree outside the diagonal, we need the Jordan totient function: 
$$J_2(k) :=k^2 \prod_{\substack{p | k \\ p \text{ prime}}}(1-\frac{1}{p^2}) .$$

\begin{theorem}\label{thm:integer}
The matrix $A_n + g_J(n)I$ has an integer completely positive factorization for $g_J(n):= \sum_{k=1}^{n-1} J_2(k)$.
\end {theorem}

\begin{proof}
The claim will follow from Lemma \ref{lem:Ei}, after we show that:
 $$A_n + g_J(n)I_n=\sum_{i=1}^{n-1}J_2(i)E_i.$$
Clearly, the equality holds for the elements on the main diagonal. The equality for off-diagonal elements follows from the following well known formula, that can be found for example in \cite{MR1042491}: 
$$ \sum_{d \text{ divisor of } k} J_2(d) = k^2.$$
\end {proof}


We know that, $g_J(n)$ is not optimal, since, for example, $g_J(6)=48>36$, and $36$ is the optimal choice for $n=6$ by Proposition \ref{prop:smalln}. 
On the other hand $g_J(n) <g_D(n)$, where $g_D(n):=\frac{1}{6}n(n-1)(2n-1)$ is the smallest value that makes $A_n+g_D(n)I_n$ diagonally dominant.  Asymptotically we have 
$$g_J(n)\sim\displaystyle\frac{n^3}{3\zeta(3)},$$ 
where $\zeta$ is Riemann zeta function, see for example \cite{MR1042491}.
This gives us: $$g_J(n) \sim \displaystyle\frac{2}{\zeta(3)}f(n) \approx 1.66381f(n),$$
which is again an improvement over $g_D(n) \sim 2 f(n).$ 

\bibliographystyle{amsplain}
\bibliography{cp}

\providecommand{\bysame}{\leavevmode\hbox to3em{\hrulefill}\thinspace}
\providecommand{\MR}{\relax\ifhmode\unskip\space\fi MR }
\providecommand{\MRhref}[2]{%
  \href{http://www.ams.org/mathscinet-getitem?mr=#1}{#2}
}
\providecommand{\href}[2]{#2}
\begin{thebibliography}{10}

\bibitem{MR1042491}
Sukumar~Das Adhikari and A.~Sankaranarayanan, \emph{On an error term related to
  the {J}ordan totient function {$J_k(n)$}}, J. Number Theory \textbf{34}
  (1990), no.~2, 178--188. \MR{1042491}

\bibitem{MR3887551}
Abdo~Y. Alfakih, \emph{Euclidean distance matrices and their applications in
  rigidity theory}, Springer, Cham, 2018. \MR{3887551}

\bibitem{MR2563025}
LeRoy~B. Beasley and Thomas~J. Laffey, \emph{Real rank versus nonnegative
  rank}, Linear Algebra Appl. \textbf{431} (2009), no.~12, 2330--2335.
  \MR{2563025}

\bibitem{MR923678}
Abraham Berman and Robert Grone, \emph{Bipartite completely positive matrices},
  Math. Proc. Cambridge Philos. Soc. \textbf{103} (1988), no.~2, 269--276.
  \MR{923678}

\bibitem{MR1298430}
Abraham Berman and Robert~J. Plemmons, \emph{Nonnegative matrices in the
  mathematical sciences}, Classics in Applied Mathematics, vol.~9, Society for
  Industrial and Applied Mathematics (SIAM), Philadelphia, PA, 1994, Revised
  reprint of the 1979 original. \MR{1298430}

\bibitem{MR1986666}
Abraham Berman and Naomi Shaked-Monderer, \emph{Completely positive matrices},
  World Scientific Publishing Co., Inc., River Edge, NJ, 2003. \MR{1986666}

\bibitem{MR3859579}
\bysame, \emph{Completely positive matrices: real, rational, and integral},
  Acta Math. Vietnam. \textbf{43} (2018), no.~4, 629--639. \MR{3859579}

\bibitem{MR1810401}
Immanuel~M. Bomze, Mirjam D\"{u}r, Etienne de~Klerk, Cornelis Roos, Arie~J.
  Quist, and Tam\'{a}s Terlaky, \emph{On copositive programming and standard
  quadratic optimization problems}, vol.~18, 2000, GO'99 Firenze, pp.~301--320.
  \MR{1810401}

\bibitem{MR2892529}
Immanuel~M. Bomze, Werner Schachinger, and Gabriele Uchida, \emph{Think
  co(mpletely)positive! {M}atrix properties, examples and a clustered
  bibliography on copositive optimization}, J. Global Optim. \textbf{52}
  (2012), no.~3, 423--445. \MR{2892529}

\bibitem{MR3023447}
Peter J.~C. Dickinson and Mirjam D\"{u}r, \emph{Linear-time complete positivity
  detection and decomposition of sparse matrices}, SIAM J. Matrix Anal. Appl.
  \textbf{33} (2012), no.~3, 701--720. \MR{3023447}

\bibitem{MR1310974}
John~H. Drew, Charles~R. Johnson, and Raphael Loewy, \emph{Completely positive
  matrices associated with {$M$}-matrices}, Linear and Multilinear Algebra
  \textbf{37} (1994), no.~4, 303--310. \MR{1310974}

\bibitem{10.1007/978-3-642-12598-0_1}
Mirjam D{\"u}r, \emph{Copositive programming -- a survey}, Recent Advances in
  Optimization and its Applications in Engineering (Berlin, Heidelberg) (Moritz
  Diehl, Francois Glineur, Elias Jarlebring, and Wim Michiels, eds.), Springer
  Berlin Heidelberg, 2010, pp.~3--20.

\bibitem{MR3624664}
Mathieu Dutour~Sikiri\'{c}, Achill Sch\"{u}rmann, and Frank Vallentin,
  \emph{Rational factorizations of completely positive matrices}, Linear
  Algebra Appl. \textbf{523} (2017), 46--51. \MR{3624664}

\bibitem{MR2964717}
Nicolas Gillis and Fran\c{c}ois Glineur, \emph{On the geometric interpretation
  of the nonnegative rank}, Linear Algebra Appl. \textbf{437} (2012), no.~11,
  2685--2712. \MR{2964717}

\bibitem{MR787367}
J.~C. Gower, \emph{Properties of {E}uclidean and non-{E}uclidean distance
  matrices}, Linear Algebra Appl. \textbf{67} (1985), 81--97. \MR{787367}

\bibitem{MR2978290}
Roger~A. Horn and Charles~R. Johnson, \emph{Matrix analysis}, second ed.,
  Cambridge University Press, Cambridge, 2013. \MR{2978290}

\bibitem{MR3152147}
Boris Horvat, Ga\v{s}per Jakli\v{c}, Iztok Kavkler, and Milan Randi\'{c},
  \emph{Rank of {H}adamard powers of {E}uclidean distance matrices}, J. Math.
  Chem. \textbf{52} (2014), no.~2, 729--740. \MR{3152147}

\bibitem{MR910984}
M.~Kaykobad, \emph{On nonnegative factorization of matrices}, Linear Algebra
  Appl. \textbf{96} (1987), 27--33. \MR{910984}

\bibitem{MR1217760}
Natalia Kogan and Abraham Berman, \emph{Characterization of completely positive
  graphs}, vol. 114, 1993, Combinatorics and algorithms (Jerusalem, 1988),
  pp.~297--304. \MR{1217760}

\bibitem{MR3904097}
Thomas~J. Laffey and Helena \v{S}migoc, \emph{Integer completely positive
  matrices of order two}, Pure Appl. Funct. Anal. \textbf{3} (2018), no.~4,
  633--638. \MR{3904097}

\bibitem{MR2653832}
Matthew~M. Lin and Moody~T. Chu, \emph{On the nonnegative rank of {E}uclidean
  distance matrices}, Linear Algebra Appl. \textbf{433} (2010), no.~3,
  681--689. \MR{2653832}

\bibitem{MR3918551}
Yaroslav Shitov, \emph{Euclidean distance matrices and separations in
  communication complexity theory}, Discrete Comput. Geom. \textbf{61} (2019),
  no.~3, 653--660. \MR{3918551}

\bibitem{MR3152068}
Julia Sponsel and Mirjam D\"{u}r, \emph{Factorization and cutting planes for
  completely positive matrices by copositive projection}, Math. Program.
  \textbf{143} (2014), no.~1-2, Ser. A, 211--229. \MR{3152068}

\end{thebibliography}

%
%
%

\end{document}